\documentclass[a4paper,10pt,leqno]{article}
\usepackage{amsmath}
\usepackage{amssymb}
\usepackage{amsthm}
\usepackage{mathrsfs}
\usepackage{makeidx}
\usepackage{comment}
\usepackage{authblk}
\usepackage{hyperref}
\usepackage[nottoc]{tocbibind}
\hypersetup{colorlinks=false, linkcolor=red, citecolor= red, urlcolor=red}
\allowdisplaybreaks[1]

\newcommand{\nat}{\mathbb{N}}			
\newcommand{\real}{\mathbb{R}}			

\newcommand{\vdim}{\mathbf{n}}			
\newcommand{\avdim}{n}				
\newcommand{\codim}{\mathbf{k}}			
\newcommand{\acodim}{k}				
\newcommand{\fulldim}{\mathbf{n}+\mathbf{k}}	
\newcommand{\afulldim}{n+k}			

\newcommand{\Var}{\mathbf{V}}			
\newcommand{\RVar}{\mathbf{RV}}			
\newcommand{\IVar}{\mathbf{IV}}			
\newcommand{\appTan}{\mathbf{T}}		

\newcommand{\Lp}[1]{\mathbf{L}^{#1}}

\newcommand{\indRad}[1]{\|#1 \|}		
\newcommand{\fvar}[1]{\partial #1}		
\newcommand{\tvar}[1]{\|\fvar{#1}\|}		
\newcommand{\genTan}{\mathbf{T}}		
\newcommand{\genTanx}[2]{\mathbf{T}(#1,#2)}	
\newcommand{\genMCV}{\mathbf{H}}		

\newcommand{\cspace}{\mathcal{C}}				
\newcommand{\ccspace}{\mathcal{C}_{\mathrm{c}}}			
\newcommand{\dspace}[1]{\mathcal{C}^{#1}}			
\newcommand{\cdspace}[1]{\mathcal{C}_{\mathrm{c}}^{#1}}		

\newcommand{\LM}[1]{\mathscr{L}^{#1}}		
\newcommand{\HM}[1]{\mathscr{H}^{#1}}		

\newcommand{\grass}{\mathbf{G}}
\newcommand{\grassn}[1]{\mathbf{G}_{#1}}

\newcommand{\projT}[1]{#1_{\natural}}

\newcommand{\density}{\boldsymbol{\Theta}}		
\newcommand{\Kronecker}{\boldsymbol{\delta}}		

\newcommand{\oball}[2]{\mathbf{U}(#1,#2)}
\newcommand{\noball}[3]{\mathbf{U}^{#1}(#2,#3)}

\newcommand{\unitmeasure}[1]{\boldsymbol{\omega}(#1)}

\newcommand{\ud}{\ensuremath{\,\mathrm{d}}}

\newcommand{\mres}{\mathop{\vrule height 1.3ex depth 0pt width
0.12ex\vrule height 0.11ex depth 0pt width 1.1ex}}

\newcommand{\Bdry}[1]{\partial #1}

\newcommand{\tder}{\partial_t}							
\newcommand{\utder}{\overline{\partial}_t}		
\newcommand{\lefttder}{\partial^-_t}		

\newcommand{\GenVar}{\mathscr{V}}		
\newcommand{\BVar}{\mathscr{B}}			
\newcommand{\TBVar}{\mathscr{T}}		
\newcommand{\force}{u}				
\newcommand{\aforce}{f}				

\newcommand{\gensubVar}{\mathcal{U}}		
\newcommand{\subVar}{\mathcal{V}}		
\newcommand{\badtimes}{J}			

\DeclareMathOperator{\spt}{spt}     

\swapnumbers
\theoremstyle{plain}
\newtheorem{thm}{Theorem}[section]

\newtheorem{lem}[thm]{Lemma}
\newtheorem{prop}[thm]{Proposition}
\newtheorem{cor}[thm]{Corollary}
\theoremstyle{definition}
\newtheorem{defn}[thm]{Definition}

\newtheorem{rem}[thm]{Remark}
\newtheorem{sett}[thm]{Setting}
\newtheorem*{note}{Note}

\numberwithin{equation}{section}

\begin{document}

\title{Equality of the usual definitions of Brakke flow}

\author{Ananda Lahiri\footnote{
Max Planck Institute for Gravitational Physics 
(Albert Einstein Institute)}}

\maketitle

%
%
\begin{abstract}
In 1978 Brakke introduced the mean curvature flow in the setting of geometric measure theory.
There exist multiple variants of the original definition.
Here we prove that most of them are indeed equal.
One central point is to correct the proof of Brakke's \S 3.5,
where he develops an estimate for the evolution of the measure of time-dependent test functions.
\end{abstract}

\tableofcontents

\section{Introduction}
\label{introduction}
\paragraph{Overview.}

In 1978 Brakke introduced the mean curvature flow in his pioneering work \cite{MR485012}.
There he uses the setting of geometric measure theory to formulate the evolution equation.
Starting with Huisken's work 
(see \cite{MR772132}) the smooth mean curvature flow came more into focus.
Around 1991 the weak mean curvature flow was reformulated as a level set problem 
by Chen, Giga, and Goto \cite{MR1100211} as well as by Evans and Spruck \cite{MR1100206}.
Later, in 2010 Liu, Sato, and Tonegawa introduced mean curvature flow with an additional translation term \cite{MR2652019}.
Some analysis on different weak formulations,
can be found in a work by Ilmanen \cite{MR1196160}
which is mainly about elliptic regularization,
but is also a great introduction to weak mean curvature flow.
In particular Ilmanen compares Brakke's definition with the level-set formulation.

Here we are interested in weak mean curvature flow in Brakke's original sense with or without additional translation term. 
Such flows will be called Brakke flows below.
Looking at the literature one finds varying definitions of Brakke flow
(see \hyperlink{Brakke_flow}{Definition of Brakke flow} (B1)- (B5)).
Basically these definitions differ in the time-dependency of the test function and/or
in that the inequality is pointwise or integral in time.

The \hyperlink{main_result}{Main Result} of this article is that these definitions are actually all equivalent
if the translation term is continuous, in particular if it is absent.
One problem is to obtain the pointwise inequality from the integral one.
This primarily follows from the upper-semi-continuity of the Brakke variation
shown by Ilmanen \cite[\S 7]{MR1196160}.

The main problem addressed here is to go from time-independent to time-dependent test-functions.
Brakke's work already contains a corresponding result \cite[\S 3.5]{MR485012}
but the proof he gives includes a major error (see Remark \ref{Brakke_error}). 
Mostly we re-arrange his calculations from \cite[ch.\ 3]{MR485012} to fix this.

\paragraph{General Assumptions.}
\hypertarget{general_assumptions}{}
\label{general_assumptions}
We consider the following situation:
\begin{itemize}
\item
$\vdim,\codim\in\nat$, 
$U\subset\real^{\fulldim}$ open, $t_1\in\real$, $t_2\in (t_1,\infty)$\\
and $(V_t)_{t\in [t_1,t_2]}$ a family in $\Var_{\vdim}(U)$.
\item
$\subVar\subset\Var_{\vdim}(U)$
such that $V_t\in\subVar$ for all $t\in [t_1,t_2]\setminus\badtimes_0$, $\LM{1}(\badtimes_0)=0$.
\item
$\BVar :\Var_{\vdim}(U)\times\cdspace{2}(U)\times(U\to\real^{\fulldim})\to [-\infty,\infty)$
Brakke variation, 
\\see Definition \ref{Variation_defs} and Definition \ref{Tonegawas_variation}.
\item
$\force:[t_1,t_2]\to(U\to\real^{\fulldim}).$ 
\end{itemize}
\begin{note}
The most interesting case is $\subVar=\IVar_{\vdim}(U)$.
\end{note}
\begin{note}
We tried to keep Definition \ref{Variation_defs} general 
to cover all Brakke variations used in the literature.
Note that for $V\in\IVar_{\vdim}(U)$ 
all established Brakke variations coinside with $\TBVar$ from Definition \ref{Tonegawas_variation}.
\end{note}

%
%
\paragraph{Definition Of Brakke Flow.}
\hypertarget{Brakke_flow}{}
The Family $(V_t)_{t\in[t_1,t_2]}$
is called a \emph{Brakke flow} if one of the following holds:
\renewcommand{\labelenumi}{(B\arabic{enumi})}
\begin{enumerate}
\item
\label{originalBrakkeflow}
For all $s\in [t_1,t_2]$ and $\phi\in\cdspace{2}(U,\real^+)$
we have
\begin{align*}
\utder\indRad{V_t}(\phi)|_{t=s}
\leq\BVar(V_s,\phi,\force_s). 
\end{align*}
\item
\label{strongBrakkeflow}
For all $s\in [t_1,t_2]$ and $\phi\in\dspace{1}([t_1,t_2],\cdspace{2}(U,\real^+))$
we have
\begin{align*}
\utder(\indRad{V_t}(\phi_t))|_{t=s}
\leq\BVar(V_s,\phi_s,\force_s)+\indRad{V_s}(\tder\phi_t|_{t=s}). 
\end{align*}
\item
\label{Tonegawaflow_variant}
For all $s_1,s_2\in[t_1,t_2]$ with $s_1<s_2$
and all $\phi\in\dspace{1}([s_1,s_2],\cdspace{2}(U,\real^+))$
such that 
$\sup_{s\in [s_1,s_2]}\BVar(V_s,\phi_s,\force_s)+\indRad{V_s}(\tder\phi_t|_{t=s})<\infty$ holds
we have that $s\to\BVar(V_s,\phi_s,\force_s)+\indRad{V_s}(\tder\phi_t|_{t=s})$ is in $\Lp{1}([s_1,s_2])$ and
\begin{align*}
\indRad{V_{s_2}}(\phi_{s_2})-\indRad{V_{s_1}}(\phi_{s_1})
\leq\int_{s_1}^{s_2}\Big(\BVar(V_s,\phi_s,\force_s)+\indRad{V_s}(\tder\phi_t|_{t=s})\Big)\ud s.
\end{align*}
\item
\label{Tonegawaflow}
For all $s_1,s_2\in[t_1,t_2]$ with $s_1<s_2$ and all $\phi\in\dspace{1}([s_1,s_2],\cdspace{2}(U,\real^+))$
we have that $s\to\BVar(V_s,\phi_s,\force_s)+\indRad{V_s}(\tder\phi_t|_{t=s})$ is in $\Lp{1}([s_1,s_2])$ and
\begin{align*}
\indRad{V_{s_2}}(\phi_{s_2})-\indRad{V_{s_1}}(\phi_{s_1})
\leq\int_{s_1}^{s_2}\Big(\BVar(V_s,\phi_s,\force_s)+\indRad{V_s}(\tder\phi_t|_{t=s})\Big)\ud s.
\end{align*}
\item
\label{weakBrakkeflow}
For all $s_1,s_2\in[t_1,t_2]$ with $s_1<s_2$ and all $\phi\in\cdspace{2}(U,\real^+)$
we have that $s\to\BVar(V_s,\phi,\force_s)$ is in $\Lp{1}([s_1,s_2])$ and
\begin{align*}
\indRad{V_{s_2}}(\phi)-\indRad{V_{s_1}}(\phi)
\leq\int_{s_1}^{s_2}\BVar(V_s,\phi,\force_s)\ud s.
\end{align*}
\end{enumerate}
\begin{note}
(B\ref{strongBrakkeflow}) implies (B\ref{originalBrakkeflow}).
(B\ref{Tonegawaflow}) implies (B\ref{Tonegawaflow_variant}).
(B\ref{Tonegawaflow}) implies (B\ref{weakBrakkeflow}).
\end{note}

%
%
\paragraph{Main Result.}
\hypertarget{main_result}{}
\emph{
Suppose $\BVar$ is upper-continuous (see Definition \ref{Variation_defs})
and $\force$ is in $\dspace{0}([t_1,t_2],\dspace{0}(U,\real^{\fulldim}))$.
Then the characterisations (B1)-(B5) are equivalent.
}
\begin{note}
If moreover $\subVar=\IVar_{\vdim}(U)$
the flow is independent on the definition of $\BVar$ for  $V\in\Var_{\vdim}(U)\setminus\IVar_{\vdim}(U)$
(see Corollary \ref{equal_variations}).
\end{note}

\paragraph{Organisation and sketch of the proof}
We will end Section \ref{introduction} with some comments on the notation
and recalling some definitions.

In Section \ref{variation} we introduce the Brakke variation and give a proper example,
see Definition \ref{Tonegawas_variation}.
Moreover we reproduce Ilmanens proof of upper-semi-continuity from \cite[\S 7]{MR1196160}.

The point of Section \ref{barriers} is to obtain uniform bounds on the measure of compact sets,
see Lemma \ref{uniformmeasurebound}.
This is easy for definitions (B\ref{Tonegawaflow}) and (B\ref{weakBrakkeflow}),
but in the other cases we have to exploit Brakke's barrier function, see Definition \ref{barrier_function}.
For this particular function we can fix the statement of \cite[\S 3.5]{MR485012},
then following Brakke's sphere comparison \cite[\S 3.6,\,3.7]{MR485012} yields the desired uniform bounds.
This already establishes (B\ref{Tonegawaflow_variant}) implies (B\ref{Tonegawaflow}).

In Section \ref{continuity} we first note that Lemma \ref{uniformmeasurebound}
directly implies that $t\to V_t(\phi)$ cannot 'jump up'.
This is then used to show Proposition \ref{almostcontinuity},
which sais that 
if $\BVar(V_{s}, \chi, \force_s)>-\infty$ then $t\to V_t(\chi)$ is continuous in $t=s$.
Using this Proposition we can fix \cite[\S 3.5]{MR485012} in the general case,
which yields (B\ref{originalBrakkeflow}) implies (B\ref{strongBrakkeflow}).
Combining Proposition \ref{almostcontinuity} with the upper-continuity of the Brakke variation
also establishes (B\ref{weakBrakkeflow}) implies (B\ref{originalBrakkeflow})
as well as (B\ref{strongBrakkeflow}) implies (B\ref{Tonegawaflow_variant}). 
This completes the proof of the \hyperlink{main_result}{Main Result}.

In the Appendix \ref{appendix} we show growth bounds for functions with bounded upper derivatives
and an $\Lp{2}$-approximation Lemma.

\paragraph{Acknowledgements.}
I want to thank Ulrich Menne for his help and advice.
I also want to thank Felix Schulze for pointing out the error in Brakke's proof of
\cite[\S 3.5]{MR485012} to me.

%
%
%

%
%
\paragraph{Notation.}
For an excellent introduction to geometric measure theory we recommend the lecture notes by Simon \cite{MR756417}.
Here we state the most important definitions.
\begin{itemize}
%
%
\item 
We set $\real^{+}:=\{x\in \mathbb{R},x\geq 0\}$
and $\nat:=\{1,2,3,\ldots\}$.

%
%
\item
For $a\in\mathbb{R}^{\vdim+\codim}$
the values $\hat{a}\in\mathbb{R}^{\vdim}$
and $\tilde{a}\in\mathbb{R}^{\codim}$ 
are given by $a=(\hat{a},\tilde{a})$.

\item
Consider an interval $I\subset\real$,
a Banach space $B$ and an $f:I\to B$.
We often identify $f_t=f(t)$ for $t\in I$.
We set $\tder f_t|_{t=s}:=\lim_{t\to s}(f_{t}-f_s)/(t-s)$
if this exists. 
In case $B=\real$ we set for $s\in I$
\begin{align*}
\utder f_t|_{t=s}:=\limsup_{t\to s}(f_{t}-f_s)/(t-s),
\quad
\lefttder f_t|_{t=s}:=\limsup_{t\nearrow s}(f_{t}-f_s)/(t-s),
\end{align*}
which are allowed to be in $[-\infty,\infty]$.
Here we also assume $s>\inf I$ in the definition of $\lefttder f_t|_{t=s}$.

\end{itemize}
Consider $\avdim,\acodim\in\mathbb{N}$
and an open set $\Omega\subset\real^{\afulldim}$.
\begin{itemize}
%
%
\item
Let $\grass(n+k,\avdim)$ denote the space of $\avdim$-dimensional subspaces of $\real^{\afulldim}$.
For $T\in\grass(\afulldim,\avdim)$
set $T^{\bot}:=\{x\in\real^{\afulldim}:\;x\cdot v=0\quad\forall v\in T\}$.
By $\projT{T}:\real^{\afulldim}\to T$ we denote the projection onto $T$.
We may identify $\projT{T}$ with the $\afulldim\times\afulldim$-matrix $M$ that satisfies
$Mx=\projT{T}(x)$ for $x\in\real^{\afulldim}$.

%
%
\item
Consider $\avdim\times\avdim$-matrices $A=(a_{ij})_{1\leq i,j\leq\avdim}$ and $B=(b_{ij})_{1\leq i,j\leq\avdim}$.
We define $A\cdot B:=\mathrm{trace}(A^TB)=\sum_{i,j=1}^{\avdim}a_{ij}b_{ij}$.

%
%
\item
For $x_0\in\real^{\avdim}$, $y_0\in\real^{\fulldim}$ 
and $R\in (0,\infty)$ 
we set
\begin{align*}
\noball{\avdim}{x_0}{R}:=\left\{x\in\real^{\avdim}:|x-x_0|<R\right\},
\;\;\;\oball{y_0}{R}:=\noball{\fulldim}{y_0}{R},
\end{align*}

%
%
\item
Let $\LM{\avdim}$ denote the $\avdim$-dimensional Lebesque measure
and $\HM{\avdim}$ denote the $\avdim$-dimensional Hausdorff measure.
Set $\unitmeasure{\avdim}:=\LM{\avdim}(\noball{\avdim}{0}{1})$.

%
%
\item
Set $\grassn{\avdim}(\Omega):=\Omega\times\grass(\afulldim,\avdim)$.
By $\Var_{\avdim}(\Omega)$ we deonte the space of all Radon measures on $\grassn{\avdim}$.
The elements of $\Var_{\avdim}(\Omega)$ are called \emph{(general) varifolds}.
\end{itemize}
%
%
%
%
Consider a Radon measure $\mu$ on $\Omega$.
\begin{itemize}

%
%
\item
Set $\spt\mu:=\{x\in U:\;\mu(\noball{n+k}{x}{r})>0,\;\text{for all}\;r\in (0,\infty)\}$.

%
%
\item
For $x\in\Omega$ we set $\density^{\avdim}(\mu,x):=\lim_{r\searrow 0}\frac{\mu(\mathbf{B}^{n+k}(x,r))}{\unitmeasure{\avdim} r^{n}}$
supposed this limit exists. 
$\density^{\avdim}(\mu,x)$ is called the \emph{density} of $\mu$ at $x$.

\end{itemize}
%
%
%
%
Consider a varifold $V\in\Var_{\avdim}(\Omega)$ 
\begin{itemize}
%
%
\item
$V$ induces a Radon measure on $\Omega$ denoted by $\indRad{V}$, which is defined via
$\indRad{V}(A):= V(\{(x,S)\in\grassn{\avdim}(\Omega),\; x\in A\})$
for any $A\subset\Omega$.

%
%
\item
Consider $y\in\Omega$.
If there exist $\theta(y)\in (0,\infty)$ and $T\in\grass(n+k,n)$
such that
\begin{align*}
\lim_{\lambda\searrow 0}\lambda^{-n}\int_{\grass_n(\Omega)}\phi(\lambda^{-1}(x-y))\,\ud V(x,S)
=\theta(y)\int_{T}\phi(x,T)\,\ud\mathscr{H}^n(x)
\end{align*}
for all $\phi\in\ccspace(\grass_n(\Omega))$,
then we set $\genTanx{V}{y}:=T$ and call this the \emph{($n$-dimensional) approximate tangent space} of $V$ at $y$
with multiplicity $\theta(y)$.

%
%
\item
We say $V$ is \emph{$n$-rectifiable},
if the approximate tangent space exists at $\indRad{V}$-a.e.\ point $x\in\Omega$.
At these points we have $\theta(x)=\density^{n}(\indRad{V},x)$.
The set of all $n$-rectifiable varifolds is denoted by $\RVar_{n}(\Omega)$.

\item
We say $V$ is \emph{integer $n$-rectifiable},
if $V$ is $n$-rectifiable and $\density^{n}(\indRad{V},x)\in\mathbb{N}$ for $\indRad{V}$-a.e. $x\in\Omega$.
The set of all integer $n$-rectifiable varifolds is denoted by $\IVar_{n}(\Omega)$.

%
%
\item
Suppose there exists an $\genMCV\in\Lp{1}_{\mathrm{loc}}((\Omega,\indRad{V}),\real^{\fulldim})$ such that
$$\int_{\grassn{\avdim}(\Omega)}D\Psi(x)\cdot S_{\natural}\ud V(x,S):=\int_{\Omega}\Psi(x)\cdot \genMCV\ud\indRad{V}(x)$$
for all $\Psi\in\cdspace{1}(\real^{\afulldim},\real^{\afulldim})$ with $\{\Psi>0\}\subset\Omega$.
Then $\genMCV$ is called the \emph{(generalised) mean curvature vector} of $V$ in $\Omega$.
\end{itemize}

\section{Brakke Variation}
\label{variation}
For a good introduction to Brakke flow we recommend the work of Ilmanen \cite{MR1196160}.
Here we give a very general definition of Brakke variation and then present an example.
For a varifold $V$ with  mean curvature vector $\genMCV$
the first variation in $\mathrm{v}$ direction with respect to $\phi$
is given by
\begin{align*}
\partial_{\mathrm{v}}V(\phi):=\int(\projT{S}^{\bot}D\phi(x)-\phi(x)\genMCV(x))\mathrm{v}(x)\ud V(x,S)
\end{align*}
(see \cite[\S 4.9.(1)]{MR0307015}).
Basically one wants to set $\BVar(V,\phi,\aforce):=\partial_{(\genMCV+\aforce^{\bot})}V(\phi)$
whenever this is defined and $\BVar(V,\phi,\aforce):=-\infty$ otherwise.
However it is quiet common to set $\BVar(V,\phi,\aforce):=-\infty$ already 
if $V$ is unrectifiable or even if it is just not integer rectifiable,
as we do in Definition \ref{Tonegawas_variation}.
The general definition of Brakke variation below covers all this variants.
In the second part of this section we prove that the particular Brakke variation 
from Definition \ref{Tonegawas_variation} is upper-continuous
which is taken from \cite[\S 7]{MR1196160}.

\begin{defn}
\label{Variation_defs}
Consider $\GenVar:\Var_{\vdim}(U)\times\cdspace{1}(U)\times(U\to\real^{\fulldim})\to [-\infty,\infty)$.
\renewcommand{\labelenumi}{(\arabic{enumi})}
\begin{enumerate}
\item
\label{Brakke_variation}
$\GenVar$ is called a \emph{Brakke variation} if 
for all $V\in\Var_{\vdim}(U)$, $\phi\in\cdspace{2} (U,\real^+)$,
and $\aforce:U\to\real^{\fulldim}$ the following holds:\\
If $V$ has a generalised mean curvature vector $\genMCV$ in $\{\phi>0\}$
and moreover $\aforce,\genMCV\in\Lp{2}((U,\indRad{V}\mres\{\phi>0\}),\real^{\fulldim})$
then
\begin{align*}
\GenVar(V,\phi,\aforce)
\leq\int_{\grassn{\vdim}(U)}\big(\projT{S^{\bot}}D\phi(x)-\phi(x)\genMCV(x)\big)\big(\genMCV(x)+\projT{S^{\bot}}\aforce(x)\big)\ud V(x,S).
\end{align*}
Otherwise $\GenVar(V,\phi,\aforce)=-\infty$.

\item
\label{proper_continuity}
$\GenVar$ is called \emph{upper-contionuous on $\gensubVar\subset\Var_{\vdim}(U)$} if the following properties hold:
\renewcommand{\labelenumii}{(2\alph{enumii})}
\begin{enumerate}
\item
\label{upper_continuity}
Let $\chi\in\cdspace{2}(U,\real^+)$, $M\in\real^+$, $V_0\in\Var_{\vdim}(U)$, 
$\aforce_0\in\dspace{0}(U,\real^{\fulldim})$.
For $i\in\nat$ consider $V_i\in\gensubVar$ and 
$\aforce_i\in\dspace{0}(U,\real^{\fulldim})$
such that 
\begin{align}
\label{uniform_measure_bound}
&\sup_{i\in\nat\cup\{0\}}\indRad{V_i}(\{\chi>0\})\leq M,
\\
\label{measure_convergence}
&\lim_{i\to\infty}\indRad{V_i}(\psi)=\indRad{V_0}(\psi),
\quad
\text{for all } \psi\in\cdspace{0}(\{\chi>0\}),
\\
\label{force_convergence}
&\lim_{i\to\infty}\|\aforce_i-\aforce_0\|_{\dspace{0}(\{\chi>0\},\real^{\fulldim})}=0.
\end{align}
Then $\limsup_{n\to\infty}\GenVar(V_i,\chi,\aforce_i)\leq\GenVar(V_0,\chi,\aforce_0)$.
\item
\label{strong_continuity}
Let $\chi\in\cdspace{2} (U,[0,1])$, $V\in\gensubVar$,
$\aforce,\genMCV\in\Lp{2}((U,\indRad{V}\mres\{\chi>0\}),\real^{\fulldim})$, 
and suppose $\genMCV$ is 
the generalised mean curvature vector of $V$ in $\{\chi>0\}$.
For $i\in\nat\cup\{0\}$ consider $\psi_i\in\cdspace{2}(\{\chi=1\},\real^+)$
such that 
$$\lim_{i\to\infty}\|\psi_i-\psi_0\|_{\dspace{2}(U)}= 0.$$
Then $\lim_{i\to\infty}\GenVar(V,\psi_i,\aforce)=\GenVar(V,\psi_0,\aforce)$.
\end{enumerate}
\end{enumerate}
\end{defn}

%
%
\begin{lem}[{\cite[\S 3.4,\,\S 3.6]{MR485012}}]
\label{Brakkevariationbound}
Consider $V\in\Var_{\vdim}(U)$, $\phi\in\cdspace{2}(U,\real^+)$ and $\aforce:U\to\real^{\fulldim}$
with $\BVar(V,\phi,\aforce)>-\infty$.
Then the following two estimates hold:
\begin{align*}
\BVar(V,\phi,\aforce)
\leq &
-\frac{1}{4}\int_U\phi|\genMCV|^2\ud\indRad{V}
+2\int_{\{\phi>0\}}(|D\phi|^2/\phi+\phi|\aforce|^2)\ud\indRad{V}.
\\
\BVar(V,\phi,\aforce)
\leq &
\int_{\grassn{\vdim}(\{\phi>0\})}\bigg(\frac{|\projT{S}D\phi(x)|^2}{2\phi(x)}-D^2\phi(x)\cdot \projT{S}\bigg)\ud V(x,S)\\
&+\int_{U}(|D\phi||\aforce|+\phi|\aforce|^2)\ud\indRad{V}.
\end{align*}
\end{lem}
%
%
\begin{rem}[{Zheng, \cite[Lem.\ 6.6]{MR1196160}}]
\label{smooth_testfct_est}
We have $|D\phi(x)|^2\leq 2\phi(x)\sup_U|D^2\phi|$
for all $x\in U$
whenever $\phi\in\cdspace{2}(U,\real^+)$.
\end{rem}
%
%
\begin{rem}[{\cite[Thm.\ 8.1.3]{MR756417}}]
\label{rect_var_equality}
For a recitifiable varifold $V\in\RVar_{\vdim}(U)$
we have
$\int_{\grass_{\vdim}(U)}\phi(x,S)\ud V(x,S)=\int_{U}\phi(x,\genTan(V,x))\ud \indRad{V}(x)$
for every $V$-integrable function $\phi$ on $\grass_{\vdim}(U)$.
\end{rem}
%
%
\begin{proof}[Proof of Lemma \ref{Brakkevariationbound}]
For the first estimate note that on $\{\phi>0\}$ we have
by Young's inequality
\begin{align*}
\projT{S^{\bot}}D\phi\cdot\genMCV
\leq |D\phi|^2/\phi+\phi|\genMCV|^2/4,
\quad
|D\phi||\aforce|
\leq |D\phi|^2/\phi+\phi|\aforce|^2.
\end{align*}
For the second estimate calculate that on $\{\phi>0\}$ we have
\begin{align*}
-\phi|\genMCV|^2-2\projT{S^{\bot}}D\phi\cdot\genMCV
&=-\phi|\genMCV|^2-2D\phi\cdot\genMCV+2\projT{S}D\phi\cdot\genMCV\\
&=-2D\phi\cdot\genMCV-\big|\projT{S}D\phi/\sqrt{\phi}-\sqrt{\phi}\genMCV\big|^2
+|\projT{S}D\phi|^2/\phi .
\end{align*}
Then use Definition \ref{Variation_defs}(\ref{Brakke_variation})
and the characterisation of the mean curvature vector to establish the result.
\end{proof}

%
%
\begin{defn}
\label{Tonegawas_variation}
Define $\TBVar :\Var_{\vdim}(U)\times\cdspace{1}(U)\times(U\to\real^{\fulldim})\to [-\infty,\infty)$ by:
If $\phi\in\cdspace{1}(U)$, $V\mres\{\phi>0\}\in\IVar_{\vdim}(U)$,
$\aforce,\genMCV\in\Lp{2}(\indRad{V};\{\phi>0\},\real^{\fulldim})$
and $\genMCV$ is the generalised mean curvature vector of $V$ in $\{\phi>0\}$ then
\begin{align*}
\TBVar(V,\phi,\aforce)
:=\int_{U}\big(D\phi(x)-\phi(x)\genMCV(x)\big)\big(\genMCV(x)+\projT{\genTan(V,x)^{\bot}}\aforce(x)\big)\ud\indRad{V}(x).
\end{align*}
Otherwise $\TBVar(V,\phi,\aforce):=-\infty$.
\end{defn}
\begin{rem}
\label{curvature_perp}
Consider an open subset $\Omega\subset\real^{\fulldim}$ and a varifold $V\in\IVar_{\vdim}(U)$ with mean curvature vector
$\genMCV$ in $\Omega$.
Then a deep theorem of Brakke \cite[Thm.\ 5.8]{MR485012} states
that $\genMCV(x)\perp\genTan(V,x)$ for $\indRad{V}$-a.e.\ $x\in U$.
\end{rem}

%
%
%
%
\begin{prop}[{\cite[\S 7]{MR1196160}}]
\label{proper_Brakke_variation}
The $\TBVar$ from Definition \ref{Tonegawas_variation} is a Brakke variation
and is upper-continuous on $\Var_{\vdim}(U)$. 
\end{prop}
%
%
\begin{cor}
\label{changed_variation_cor}
Let $\TBVar^*$ be a Brakke Variation with 
$\TBVar^*(V,\phi,\aforce)=\TBVar(V,\phi,\aforce)$ whenever $V\in\IVar_{\vdim}(U)$.
Then $\TBVar^*$ is upper-contionuous on $\IVar_{\vdim}(U)$.
\end{cor}
%
%
\begin{proof}[Proof of Proposition \ref{proper_Brakke_variation}]
To see that $\TBVar$ is a Brakke variation
use Remarks \ref{rect_var_equality} and \ref{curvature_perp}.\\
\emph{upper-continuity (\ref{upper_continuity}):}\\
This was proven by Ilmanen \cite[\S 7]{MR1196160} for $\aforce_i\equiv 0$.
The proof can be adopted without difficulties.
For the convenience of the reader we include all the details.
Consider $\chi\in\cdspace{2}(U,\real^+)$, $M\in [1,\infty)$, 
$V_i\in\Var_{\vdim}(U)$, and $\aforce_i\in\dspace{0}(U,\real^{\fulldim})$, $i\in\nat\cup\{0\}$
as in Definition \ref{Variation_defs}(\ref{upper_continuity}).
Set 
\begin{align*}
\Omega:=\{\chi>0\},
\quad
\Lp{2}_i:=\Lp{2}((\Omega,\indRad{V_i}),R),\, i\in\nat,
\end{align*}
where $R$ is $\real$ or $\real^{\fulldim}$ which will be clear from the context.
Consider the case that $\limsup\TBVar(V_i,\chi,\aforce_i)=-\infty$, then the result folows immediately.
Thus we may assume (by taking a subsequence) that $m_0:=\inf_{i\in\nat}\TBVar(V_i,\chi,\aforce_i)>-\infty$.
In particular for all $i\in\nat$ we have $V_i\in\IVar_{\vdim}(\Omega)$ 
and there exists a mean curvature vector $\genMCV_i\in\Lp{2}_i$.
On the other hand by assumtpions \eqref{uniform_measure_bound},\eqref{force_convergence}
and Lemma \ref{Brakkevariationbound} we have
\begin{align*}
&\TBVar(V_i,\chi,\aforce_i)+\frac{1}{4}\int_U\chi|\genMCV_i|^2\ud\indRad{V_{i}}
\leq 2\int_{\Omega}(|D\chi|^2/\chi+\chi|\aforce_i|^2)\ud\indRad{V_{i}}\\
&\leq 2\|\chi\|_{\dspace{2}(U)}(1+\|\aforce_i\|^2_{\dspace{0}(U,\real^{\fulldim})})M
\leq 4\|\chi\|_{\dspace{2}(U)}(1+\|\aforce_0\|^2_{\dspace{0}(U,\real^{\fulldim})})M=:L_0
\end{align*}
for $i$ large enough. Thus
\begin{align}
\label{curvature_compactness_bound}
\int_U\chi|\genMCV_i|^2\ud\indRad{V_{i}}\leq 4(L_0-m_0)=:L<\infty
\end{align}
for $i$ large enough.
In view of assumption \eqref{measure_convergence} 
and Allards compactness theorem for integer rectifiable varifolds \cite[Thm.\ 6.4]{MR0307015}
we may assume (by taking a subsequence) that
\begin{align}
\label{varifold_convergence}
V_i\mres\Omega\to V_0\mres\Omega
\end{align}
as varifolds (Radon measures on $\grassn{\vdim}(\Omega)$)
in particular $V_0\in\IVar_{\vdim}(\Omega)$.
Also, the first variation bound from the compactness theorem \cite[Thm.\ 6.4]{MR0307015} combined with
H\"olders inequality and estiamte \eqref{curvature_compactness_bound} imply that the mean curvature vector $\genMCV_0$ of $V_0$ in $\Omega$ exists.
%
%
%

For the moment fix some $\psi\in\cdspace{2}(\Omega,\real^+)$
with $\psi\leq\chi$.
In view of \eqref{varifold_convergence} we directly have
\begin{align}
\label{BrakkeVar_convergence2}
\lim_{i\to \infty}\int_{U}D\psi\cdot\genMCV_i\ud\indRad{V_i}
=\int_{U}D\psi\cdot\genMCV_0\ud\indRad{V_0}.
\end{align}
We want to show
\begin{align}
\label{BrakkeVar_supbound}
B_1:=\int_{U}\psi|\genMCV_0|^2\ud\indRad{V_0}
\leq\liminf_{i\to\infty}\int_{U}\psi|\genMCV_i|^2\ud\indRad{V_i}=:L_{\psi}\leq L.
\end{align}
Note that $L_{\psi}\leq L$ follows directly from \eqref{curvature_compactness_bound} and $\psi\leq\chi$.
%
%
%
%
%
%
Let $\epsilon\in (0,1)$ be given.
By Lemma \ref{L2_approx_lem} with $H=\sqrt{\psi}\genMCV_0$
we find a vectorfield $X\in\dspace{1}(\{\psi>0\},\real^{\fulldim})$ with $\|X\|_{\Lp{1}_0}\leq 1$ and 
\begin{align*}
\sqrt{B_1}
\leq\int_{\Omega}\sqrt{\psi}\genMCV_0\cdot X\,\ud\indRad{V_0} + \epsilon.
\end{align*}
Then by \eqref{varifold_convergence} and H\"olders inequality
we conclude for some large enough $i$ that
\begin{align*}
\sqrt{B_1}
\leq\int_{\Omega}\sqrt{\psi}\genMCV_i\cdot X\,\ud\indRad{V_i} + 2\epsilon
\leq\sqrt{L_{\psi}}\|X\|_{\Lp{2}_i} + 3\epsilon
\leq\sqrt{L_{\psi}}(1+\epsilon) + 3\epsilon
\end{align*}
and for $\epsilon\searrow 0$ this implies estimate \eqref{BrakkeVar_supbound}.
Next we claim
\begin{align}
\label{BrakkeVar_convergence1}
\lim_{i\to\infty}B_2(i)=&B_2(0),
\\\nonumber
\text{where}\quad 
B_2(i):=&\int_{U}(D\psi-\psi\genMCV_i)\cdot\projT{\appTan(\indRad{V_i},x)^{\bot}}\aforce_i\ud\indRad{V_i}
\quad\text{for }i\in\nat\cup\{0\}.
\end{align}
To see this estimate
\begin{align*}
&|B_2(i)-B_2(0)|\\
&\leq\bigg|\int_{\grassn{\vdim}(U)}D\psi(x)\cdot\projT{S^{\bot}}\aforce_0(x)\ud V_i(x,S)
-\int_{\grassn{\vdim}(U)}D\psi(x)\cdot\projT{S^{\bot}}\aforce_0(x)\ud V_0(x,S)\bigg|
\\&\quad+\bigg|\int_{U}\psi\genMCV_i\cdot\aforce_0\ud\indRad{V_i}
-\int_{U}\psi\genMCV_0\cdot\aforce_0\ud\indRad{V_0}\bigg|
\\&\quad+\int_{U}(|D\psi|+\psi|\genMCV_i|)|\aforce_0-\aforce_i|\ud\indRad{V_i}.
\end{align*}
Here we also used Remark \ref{curvature_perp}.
Now by \eqref{uniform_measure_bound},\eqref{force_convergence},\eqref{curvature_compactness_bound},\eqref{varifold_convergence},
and H\"older`s inequality we see
\begin{align*}
\limsup_{i\to\infty}|B_2(i)-B_2(0)|^2
&\leq\limsup_{i\to\infty}(\|D\psi\|_{\Lp{2}_i}^2+\|\psi\genMCV_i\|_{\Lp{2}_i}^2)\|\aforce_0-\aforce_i\|_{\Lp{2}_i}^2\\
&\leq(\|\psi\|_{\dspace{1}(\Omega)}^2+1) (M + L)M\limsup_{i\to\infty}\|\aforce_0-\aforce_i\|_{\dspace{0}(\Omega)}^2
=0
\end{align*}
This shows \eqref{BrakkeVar_convergence1}.

Now combining \eqref{BrakkeVar_convergence2}, \eqref{BrakkeVar_supbound}, and \eqref{BrakkeVar_convergence1} we arrive at
\begin{align}
\label{BrakkeVar_bound}
\limsup_{i\to \infty}\TBVar(V_i,\psi,\aforce_i)
\leq \TBVar(V_0,\psi,\aforce_0)
\end{align}
for all $\psi\in\cdspace{2}(\Omega,\real^+)$ with $\psi\leq\chi$.
Note that by Lemma \ref{Brakkevariationbound},
and Remark \ref{smooth_testfct_est} 
combined with estimates \eqref{uniform_measure_bound} and \eqref{force_convergence}
we have
\begin{align}
\label{BrakkeVar_uniform_bound}
\TBVar(V_i,\psi,\aforce_i)
\leq
4M(1+\|\aforce_0\|_{\dspace{0}}^2)\|\psi\|_{\dspace{2}}
\end{align}
for all $\psi\in\cdspace{2}(U,\real^+)$
and all $i\in\nat$.

Next take a sequence $(\psi_j)$ in $\in\cdspace{2}(\Omega,\real^+)$
such that $\psi_j\leq\chi$ and
\begin{align}
\label{BrakkeVar_testfct_approx}
\lim_{j\to\infty}\|\psi_j-\chi\|_{\dspace{2}(\Omega)}= 0.
\end{align}
By Definition \ref{Tonegawas_variation},
estimates \eqref{uniform_measure_bound},\eqref{BrakkeVar_supbound}, \eqref{BrakkeVar_testfct_approx},
and the dominated convergence theorem we obtain
\begin{align}
\label{BrakkeVar_BVar_approx}
\lim_{j\to\infty}\TBVar(V_0,\psi_j,\aforce_0)=\TBVar(V_0,\chi,\aforce_0).
\end{align}
Using \eqref{BrakkeVar_bound} and \eqref{BrakkeVar_uniform_bound}
we can estimate
\begin{align*}
\limsup_{i\to\infty}\TBVar(V_i,\chi,\aforce_i)
&\leq\limsup_{i\to\infty}\TBVar(V_i,\psi_j,\aforce_i)
+\limsup_{i\to\infty}\TBVar(V_i,\chi-\psi_j,\aforce_i)
\\
&\leq\TBVar(V_0,\psi_j,\aforce_0)+4M(1+\|\aforce_0\|_{\dspace{0}}^2)\|\chi-\psi_j\|_{\dspace{2}}
\end{align*}
for all $j\in\nat$.
Thus in view of \eqref{BrakkeVar_testfct_approx} and \eqref{BrakkeVar_BVar_approx},
letting $j\to\infty$ establishes the desired estimate.
\\
\emph{upper-continuity (\ref{strong_continuity}):}\\
Consider $\chi\in\cdspace{2}(U,[0,1])$, $M\in [1,\infty)$, 
$V\in\gensubVar$,
$\psi_i\in\cdspace{2}(\{\chi=1\},\real^+)$, and 
$\aforce,\genMCV\in\Lp{2}((U,\indRad{V}\mres\{\chi>0\}),\real^{\fulldim})$
as in Definition \ref{Variation_defs}(\ref{strong_continuity}).
Using the integrability of $f$ and $\genMCV$
combined with Young's inequality we cobtain
\begin{align*}
|\TBVar(V,\psi_i,\aforce)-\TBVar(V,\psi_0,\aforce)|
&\leq
4\|\psi_i-\psi_0\|_{\dspace{1}}\int_{\{\chi=1\}}(1+|\genMCV|^2+|\aforce|^2)\ud\indRad{V}
\end{align*}
for all $i\in\nat$
and the statement follows immediately.
\end{proof}
\begin{proof}[proof of the Corollary]
Consider the situation of Definition \ref{Variation_defs}(\ref{upper_continuity}).
We want to show $\limsup_{n\to\infty}\TBVar^*(V_i,\chi,\aforce_i)\leq\TBVar^*(V_0,\chi,\aforce_0)$.
As all $V_i$ are in $\IVar_{\vdim}(U)$ we can proceed as in the beginning of the previous proof.
We may assume that $\inf_{i\in\nat}\TBVar^*(V_i,\chi,\aforce_i)>-\infty$.
This yields estimate \eqref{curvature_compactness_bound}.
Then Allards compactness theorem \cite[Thm.\ 6.4]{MR0307015}
implies $V_0\in\IVar_{\vdim}(\{\chi>0\})$.
In view of Proposition \ref{proper_Brakke_variation} this establishes the statement.
\end{proof}

\section{Barriers}
\label{barriers}
Here we show that for a Brakke flow the measure inside some fixed compact set is uniformly bounded in time.
This follows directly from the definition if (B\ref{Tonegawaflow}) or (B\ref{weakBrakkeflow}) holds.
For the other cases we follow Brakke \cite[\S 3.6,\,3.7]{MR485012}.
In particular we change some of his calculations that rely on \cite[\S 3.5]{MR485012}.
This yields a uniform measure bound for balls
and the statement for compact sets is a straight forward consequence.

\begin{sett}
\label{bounded_force_setting}
Consider the \hyperlink{general_assumptions}{General Assumptions}
and additionally suppose that one of the following holds:
\renewcommand{\labelenumi}{(\alph{enumi})}
\begin{enumerate}
\item
For all $K\subset\subset U$
we have $\Gamma_{a}(K):=\sup_{t\in[t_1,t_2]}\|\force_t\|_{\Lp{\infty}((K,\indRad{V_t}),\real^{\fulldim})}<\infty$.
\item
For all $K\subset\subset U$
we have $\Gamma_{b}(K):=\sup_{t\in[t_1,t_2]}\|\force_t\|_{\Lp{2}((K,\indRad{V_t}),\real^{\fulldim})}<\infty$.
\end{enumerate}
For $K\subset\subset U$ set $\Gamma(K):=\min\{\Gamma_{a}(K),\Gamma_{b}(K)\}$.
\end{sett}
\begin{sett}
\label{main_setting}
Consider Setting \ref{bounded_force_setting}
and additionally suppose $(V_t)_{t\in [t_1,t_2]}$ is a \hyperlink{Brakke_flow}{Brakke flow},
i.e.\ one of (B1)-(B5) holds.
\end{sett}
\begin{note}
\label{continuous_force_setting}
If $\force\in\dspace{0}([t_1,t_2],\dspace{0}(U,\real^{\fulldim}))$
we have that (a) holds.
\end{note}
\begin{note}
Once Lemma \ref{uniformmeasurebound} is established (a) implies (b).
\end{note}
%
%
\begin{lem}
\label{weakbarrier}
Consider Setting \ref{bounded_force_setting}
and suppose (B\ref{originalBrakkeflow}) or (B\ref{Tonegawaflow_variant}) holds.
Let $x_0\in\real^{\fulldim}$ and $R\in (0,\infty)$
such that $\oball{x_0}{2R}\subset\subset U$.
Fix $\Gamma=\Gamma(\oball{x_0}{2R})$.
Then for $s_1\in [t_1,t_2)$ and $s_2\in (s_1,t_2]$ with $s_2\leq s_1+R^2/(2\vdim+4+2R^2\Gamma^2)$
we have
\begin{align*}
&\indRad{V_{s_2}}(\oball{x_0}{R})
\leq
16(\indRad{V_{s_1}}(\oball{x_0}{2R})+R^{2}\Gamma^2),
\end{align*}
\end{lem}
The proof of this Lemma is based on the properties of the following barrier function
which was introduced by Brakke \cite[ch.\ 3]{MR485012}.
%
%
\begin{defn}[{\cite[\S 3.6]{MR485012}}]
\label{barrier_function}
Given $s_1,\Lambda_0\in\real$ and $R\in (0,\infty)$
we consider the barrier function $\varphi\in\dspace{0,1}(\real,\cdspace{0,1}(\real^{\fulldim},\real^+))$ 
given by
\begin{align*}
\varphi_t(x):=\varphi(t,x):=\max\{1-(2R)^{-2}(|x|^2+\Lambda_0 (t-s_1)),0\}.
\end{align*}
\end{defn}
%
%
\begin{note}
\begin{itemize}
\item 
For $p\in\nat$, $p\geq 2$ we have $\varphi^p\in\dspace{p-1}(\real,\cdspace{p-1}(\real^{\fulldim},\real^+))$.
\item
For $t\in [s_1,\infty)$ we have
$\{\varphi_t>0\}\subset\oball{0}{2R}$ and $\varphi_t\leq 1$.
\end{itemize}
\end{note}
%
%
\begin{lem}[{\cite[\S 3.6]{MR485012}}]
\label{barrier_variation_bound_lem}
For $s_1\in [t_1,t_2)$, $\Lambda_0\in\real$, $R\in (0,\infty)$ and $p\in\nat$, $p\geq 3$
we can estimate
\begin{align*}
\BVar(V_s,\varphi_s^p,\force_s)
\leq 
\int_{U}\varphi_s^{p-1}\left(\frac{p\vdim}{2R^2}+\frac{p^2}{4R^2}+2|\force_s|^2\right)\ud\indRad{V_t}
\end{align*}
for all $s\in [s_1,t_2]$.
\end{lem}
\begin{proof}
We may assume $\BVar(V_s,\varphi_s^P,\force)>-\infty$.
Then we can apply Lemma \ref{Brakkevariationbound} to estimate
\begin{align}
\label{barrier-Brakke_variation_bound}
\begin{split}
\BVar(V_s,\varphi_s^p,\force_s)
\leq &
\int_{\grassn{\vdim}(\{\varphi_s>0\})}\bigg(\frac{|\projT{S}D\varphi_s^p|^2}{2\varphi_s^p}-D^2\varphi_s^p(x)\cdot\projT{S}\bigg)\ud V_t(x,S)\\
&+\int_{U}(|D\varphi_s^p|^2/(4\varphi_s^{p-1})+2\varphi_s^{p-1}|\force_s|^2)\ud\indRad{V_s}.
\end{split}
\end{align}
Here we already used Young's inequality and $\varphi_s\leq 1$
for the second integral.
By definition of $\varphi$ we have
\begin{align*}
D\varphi_s^p(x)
=-\frac{p}{2R^2}\varphi_s^{p-1}(x)x,
\quad
|\projT{S}D\varphi_s^p(x)|^2
=\frac{p^2}{4R^4}\varphi_s^{2p-2}(x)|\projT{S}x|^2,
\\
(D^2\varphi_s^p(x))_{ij}
=-\frac{p}{2R^2}\varphi_s^{p-1}(x)\Kronecker_{ij}
+\frac{p(p-1)}{4R^4}\varphi_s^{p-2}(x)x_ix_j,
\\
D^2\varphi_s^p(x)\cdot\projT{S}
=-\frac{p\vdim}{2R^2}\varphi_s^{p-1}(x)
+\frac{p(p-1)}{4R^4}\varphi_s^{p-2}(x)|\projT{S}x|^2
\end{align*}
for all $(x,S)\in\grassn{\vdim}(\{\varphi_s>0\})$.
Inserting these equations into \eqref{barrier-Brakke_variation_bound} establish the statement.
\end{proof}

%
%
\begin{proof}[{Proof of Lemma \ref{weakbarrier}}]
We may assume $x_0=0$ and $s_1=0$. 
Consider $\varphi$ from above
for $\Lambda_0:=2\vdim+4+2R^2\Gamma^2$.
Note that $\tder\varphi_t^4=-\Lambda_0R^{-2}\varphi_t^3$.
We observe that it suffices to prove
\begin{align}
\label{barrier_function_final_estimate}
\indRad{V_t}(\varphi^4_t)\leq \indRad{V_0}(\varphi^4_0)+ 2\Gamma^2 t
\end{align}
for all $t\in [0,t_2]$.
Then the result immediately follows from
$\{\varphi_0>0\}\subset\oball{0}{2R}$ and
$\inf_{t\in [0,R^2/\Lambda_0]}\inf_{\oball{0}{R}}\varphi_t\geq 1/2$.

\textit{Case 1:} (B\ref{Tonegawaflow_variant}) holds:\\
The inequality in (B\ref{Tonegawaflow_variant}) and Lemma \ref{barrier_variation_bound_lem}
yield
\begin{align*}
\indRad{V_t}(\varphi^4_t)-\indRad{V_0}(\varphi^4_0)
\leq 
2\int_{0}^{t}\int_{U}\varphi_s^{3}\left(|\force_s|^2-\Gamma^2\right)\ud\indRad{V_s}\ud s
\end{align*}
for all $t\in (0,t_2]$.
Then by definition of $\Gamma$ (see Setting \ref{bounded_force_setting}) 
and $\varphi_s\leq 1$
we obtain that estimate \eqref{barrier_function_final_estimate} holds.

\textit{Case 2:} (B\ref{originalBrakkeflow}) holds:\\
We want to show the following bound on the upper derivative from the left
\begin{align}
\label{leftupperderivativebound}
\lefttder\indRad{V_t}(\varphi^4_t)|_{t=s}
\leq\BVar(V_s,\varphi_s^4,\force_s)
-\Lambda_0 R^{-2}\indRad{V_s}(\varphi_s^3)
\end{align}
for all $s\in (0,t_2]$.
To see this fix $s\in (0,t_2]$ and for $\delta\in (0,s)$ consider
\begin{align*}
\xi_s(\delta):=\delta^{-1}(\indRad{V_{s}}(\varphi^4_{s})-\indRad{V_{s-\delta}}(\varphi^4_{s-\delta})).
\end{align*}
By the monotonicity of $\tder\varphi^4_t$ and the mean value formula we can estimate
${\varphi^4_{s}-{\varphi^4_{s-\delta}\leq\delta\tder\varphi^4_{t}|_{t=s}.}=-\delta\Lambda_0 R^{-2}\varphi_s^3.}$
Thus
\begin{align}
\label{leftupperderivativebound2}
\xi_s(\delta)\leq\delta^{-1}(\indRad{V_{s}}(\varphi^4_{s})-\indRad{V_{s-\delta}}(\varphi^4_{s}))
-\Lambda_0 R^{-2}\indRad{V_{s-\delta}}(\varphi_s^3).
\end{align}
By (B\ref{originalBrakkeflow}) we have
$\utder\indRad{V_{t}}(\phi)|_{t=s}
\leq\BVar(V_s,\phi,\force)<\infty$.
This yields
\begin{align*}
\limsup_{\delta\searrow 0} \delta^{-1}(\indRad{V_{s}}(\varphi^4_{s})-\indRad{V_{s-\delta}}(\varphi^4_{s}))
&\leq\utder\indRad{V_{t}}(\varphi^4_{s})|_{t=s}
\leq\BVar(V_s,\varphi_s^4,\force_s)
\\
\liminf_{\delta\searrow 0}\indRad{V_{s-\delta}}(\varphi_s^3)
&\geq\indRad{V_{s}}(\varphi_s^3).
\end{align*}
Now taking $\limsup_{\delta\searrow 0}$ in \eqref{leftupperderivativebound2} implies \eqref{leftupperderivativebound}.

Considering the definition of $\Lambda_0$
we can combine inequality \eqref{leftupperderivativebound}
with Lemma \ref{barrier_variation_bound_lem} to arrive at
\begin{align*}
\lefttder\indRad{V_t}(\varphi^4_t)|_{t=s}
\leq 
2\int_{U}\big(|\force_s|^2-\Gamma^2\big)\varphi_s^3\ud\indRad{V_s}.
\end{align*}
By definition of $\Gamma$ (see Setting \ref{bounded_force_setting})
and by $\varphi_s\leq 1$ this yields
\begin{align*}
\lefttder\indRad{V_t}(\varphi^4_t)|_{t=s}\leq 2\Gamma^2.
\end{align*}
As $s\in (0,t_2]$ was arbitrary we can now use Lemma \ref{leftmonotonicitylem}
to obtain estimate \eqref{barrier_function_final_estimate}
and thus the result.
\end{proof}
\begin{lem}
\label{uniformmeasurebound} 
In Setting \ref{main_setting} the following holds:
For all $K\subset\subset U$ there exist an $M\in\real^+$ such that
\renewcommand{\labelenumi}{(\arabic{enumi})}
\begin{enumerate}
\item 
$\sup_{t\in [t_1,t_2]}\indRad{V_t}(K)\leq M$.
\item
$\sup_{t\in [t_1,t_2]}\BVar(V_t,\phi,\force_t)\leq M\|\phi\|_{\dspace{2}(U)}$
for all $\phi\in\cdspace{2}(K,\real^+)$.
\end{enumerate}
\end{lem}
\begin{proof}[Proof of (1)]
\textit{Case 1:} (B\ref{weakBrakkeflow}) or (B\ref{Tonegawaflow}) holds:\\
There exists an $\phi\in\cdspace{2}(U,[0,1])$ with $K\subset\{\phi=1\}$.
Then we can estimate
\begin{align*}
\indRad{V_s}(K)
&\leq\indRad{V_s}(\phi)
\leq\indRad{V_{t_1}}(\phi)+\int_{t_1}^{s}\BVar(V_t,\phi,\force_t)\ud t\\
&\leq\indRad{V_{t_1}}(\phi)+\int_{t_1}^{t_2}\max\{\BVar(V_t,\phi,\force_t),0\}\ud t
=:M<\infty
\end{align*}
for all $s\in [t_1,t_2]$.
Here we used that $s\to\BVar(V_s,\phi,\force_s)$ is in $\Lp{1}$.

\textit{Case 2:} (B\ref{originalBrakkeflow}) or (B\ref{strongBrakkeflow}) or (B\ref{Tonegawaflow_variant}) holds:\\
Set $R:=\inf\{|x-y|, x\in K, y\in\Bdry{U}\}/4>0$.
Then we find $P\in\nat$ and $x_1,\ldots,x_P\in K$ 
such that $K\subset\bigcup_{i=1}^P\oball{x_i}{R}$.
Fix some $i\in \{1,\ldots,P\}$.
By Lemma \ref{weakbarrier} there exist $\Gamma_i\in [0,\infty)$
such that
\begin{align*}
\indRad{V_{s}}(\oball{x_i}{R})\leq 16(\indRad{V_{s_0}}(\oball{x_i}{2R})+\Gamma_i)=:M_i(s_0)<\infty
\end{align*}
for all $s,s_0\in [t_1,t_2]$ with $0\leq s-s_0\leq R^2/(2\vdim+4+\Gamma_i^2)$.
Set $\tau_i:= R^2/(2\vdim+4+\Gamma_i^2)$ and $M_i:=\max\{M_i(j\tau_i),j\in\nat\cup\{0\}, j\tau_i\leq t_2-t_1\}$.
Then $\indRad{V_{s}}(\oball{x_i}{R})\leq M_i$ for all $s\in [t_1,t_2]$
and the result follows with $M:=\sum_{i=1}^PM_i<\infty$.
\end{proof}
\begin{proof}[Proof of (2)]
By (1) we have $\sup_{t\in[t_1,t_2]}\indRad{V_t}(K)=: M_0\in\real^+$.
Consider some $\phi\in\cdspace{2}(K,\real^+)$. Then Lemma \ref{Brakkevariationbound} yields
\begin{align*}
\BVar(V_t,\phi,\force_t)
\leq
\begin{cases}
2\|\phi\|_{\dspace{2}(U)}(1+(\Gamma_a(K))^2)M_0
&\text{if \ref{bounded_force_setting}(a) holds}
\\
2\|\phi\|_{\dspace{2}(U)}(M_0+(\Gamma_b(K))^2)
&\text{if \ref{bounded_force_setting}(b) holds}
\end{cases}
\end{align*}
for all $t\in [t_1,t_2]$.
\end{proof}

%
%
\begin{cor}
\label{implications_of_weak_definitions}
In Setting \ref{bounded_force_setting} we have (B\ref{Tonegawaflow_variant}) implies (B\ref{Tonegawaflow}).
\end{cor}
\begin{proof}
Suppose (B\ref{Tonegawaflow_variant}) holds.
In particular we are in Setting \ref{main_setting}.
We consider arbitrary $t_1\leq s_1<s_2\leq t_2$ and $\psi\in\dspace{1}([s_1,s_2],\cdspace{2}(U,\real^+))$
in particular 
$\sup_{s\in\real}\|\psi_s\|_{\dspace{2}(U)}+\|\tder\psi_t|_{t=s}\|_{\cspace(U)}<\infty$.
Define $B:[s_1,s_2]\to\real$ by
$$B(s):=\BVar(V_s,\psi_s,\force_s)+\indRad{V_s}(\tder\psi_t|_{t=s}).$$
By Lemma \ref{uniformmeasurebound} we have
$\sup_{s\in [s_1,s_2]}B(s)<\infty$.
Then (B\ref{Tonegawaflow_variant}) implies $B\in\Lp{1}([s_1,s_2])$
and we conclude (B\ref{Tonegawaflow}).
\end{proof}

\section{Continuity Properties}
\label{continuity}
The uniform bounds from the previous section 
directly yield some semi-continuity for Brakke flows.
Using this we prove that at times where the Brakke variation is finite
the flow is continuous.
Once these continuity properties are established
we can proof Brakke's \cite[\S 3.5]{MR485012} and conclude the main result.

%
%
\begin{rem}
\label{upper_origianl_BF}
Consider Setting \ref{bounded_force_setting}
and suppose (B\ref{weakBrakkeflow}) holds.
Let $s\in [t_1,t_2]$, $\phi\in\cdspace{1}(U,\real^+)$,
and $\badtimes\subset [t_1,t_2]$ with $\LM{1}(\badtimes)=0$.
Then
\begin{align*}
\utder\indRad{V_t}(\phi)|_{t=s}
\leq\lim_{\delta\to 0}\sup_{t\in (s-\delta,s+\delta)\cap[t_1,t_2]\setminus\badtimes}\BVar(V_t,\phi,\force_t). 
\end{align*}
\end{rem}

%
%
\begin{prop}[{\cite[\S 3.10]{MR485012}}]
\label{semicontinuity}
In Setting \ref{main_setting} the following holds:
\renewcommand{\labelenumi}{(\arabic{enumi})}
\begin{enumerate}
\item 
For every $\phi\in\cdspace{2}(U,\real^+)$
we have $\sup_{s\in [t_1,t_2]}\utder\indRad{V_t}(\phi)|_{t=s}<\infty$.
\item
For every $\phi\in\cdspace{0}(U,\real^+)$ we have 
\begin{align*}
\lim_{\delta\searrow 0}\indRad{V_{s+\delta}}\left(\phi\right)
\leq\indRad{V_s}\left(\phi\right)
\leq\lim_{\delta\searrow 0}\indRad{V_{s-\delta}}\left(\phi\right)
\quad\text{for all }s\in [t_1,t_2].
\end{align*}
\end{enumerate}
\end{prop}
%
%
\begin{proof}
{\it For Statement (1)}
we set $L:=\sup_{t\in [t_1,t_2]}\BVar(V_t,\phi,\force_t)<\infty$,
where we used Lemma \ref{uniformmeasurebound}.
In case (B\ref{originalBrakkeflow}) or (B\ref{strongBrakkeflow}) holds, this directly implies the first statement.
If (B\ref{Tonegawaflow_variant}) or (B\ref{Tonegawaflow}) or (B\ref{weakBrakkeflow}) holds, 
first note that by Corollary \ref{implications_of_weak_definitions}
always (B\ref{weakBrakkeflow}) holds.
Then use Remark \ref{upper_origianl_BF} to conclude Statement (1).

{\it For Statement (2)} first consider $\phi\in\cdspace{2}(U,\real^+)$.
By Statement (1) we can set $L:=\sup_{s\in [t_1,t_2]}\utder\indRad{V_t}(\phi)|_{t=s}\in\real$
and $f(t):=\indRad{V_t}(\phi)-Lt$.
Then Lemma \ref{leftmonotonicitylem} yields that $f$ is monotonically non-increasing
which implies the desired estimate.
In view of Lemma \ref{uniformmeasurebound} we can use an approximation argument
to obtain Statement (2) for all $\phi\in\cdspace{0}(U,\real^+)$.
\end{proof}

%
%
\begin{prop}
\label{almostcontinuity}
In Setting \ref{main_setting} the following holds:
Consider $\chi\in\cdspace{0}(U,\real^+)$, $\psi\in\cdspace{0}(U)$, $s\in [t_1,t_2]$.
Suppose $\{|\psi|>0\}\subset\{\chi>0\}$
and $\utder\indRad{V_t}(\chi)|_{t=s}>-\infty$.
Then $\lim_{\delta\to 0}\indRad{V_{s+\delta}}(\psi)=\indRad{V_s}(\psi)$.
\end{prop}
%
%
\begin{proof}
Define the linear functional $L:\cdspace{0}(U)\to\real$ by
\begin{align*}
L(\phi):=\lim_{\delta\searrow 0}\indRad{V_{s-\delta}}(\phi)-\lim_{\delta\searrow 0}\indRad{V_{s+\delta}}(\phi).
\end{align*}
Here we set $\lim_{\delta\searrow 0}\indRad{V_{t_1-\delta}}(\phi):=\indRad{V_{t_1}}(\phi)$
and $\lim_{\delta\searrow 0}\indRad{V_{t_2+\delta}}(\phi):=\indRad{V_{t_2}}(\phi)$.
To see that the limits above indeed exist
write $\phi=\max\{\phi,0\}-\max\{-\phi,0\}$
and use Proposition \ref{semicontinuity}.
Consider a compact subset $K\subset U$ and $\phi\in\cdspace{0}(U,[-1,1])$ with $\spt\phi\subset K$.
Then by Lemma \ref{uniformmeasurebound} we can estimate
\begin{align*}
L(\phi)
\leq 2 \sup_{t\in [t_1,t_2]}\indRad{V_t}(|\phi|)
\leq 2 \sup_{t\in [t_1,t_2]}\indRad{V_t}(K)
<\infty.
\end{align*}
Also, by Proposition \ref{semicontinuity} we see that $L(\phi)\geq 0$ for $\phi\geq 0$.
Then by Riesz Representation Theorem \cite[Thm.\ 1.4.1, Rem.\ 1.4.3]{MR756417} 
there exists a Radon measure $\mu$ on $U$ such that
\begin{align}
\label{jumpingmeasure}
L(\phi)
=\int_U\phi\ud\mu
\quad\text{for all }\phi\in\cdspace{0}(U).
\end{align}
By Lemma \ref{uniformmeasurebound} we can set $M:=\sup_{t\in[t_1,t_2]}\indRad{V_t}(\{\chi>0\})+1<\infty$.
We consider two cases:

\textit{ Case 1:}
$\spt\mu\cap\{|\psi|>0\}=\emptyset$.
\\
Let $\epsilon>0$ be given.
Consider $\psi^{+}:=\max\{0,\psi\}$, $\psi^{-}:=\max\{0,-\psi\}$.
By case-assumption and equation \eqref{jumpingmeasure} we see $L(\psi^{\pm})=0$
so there exists a $\delta_0=\delta_0(\epsilon)>0$ such that for all 
$s_1\in [s-\delta_0,s)\cap [t_1,t_2]$, $s_2\in(s,s+\delta_0]\cap [t_1,t_2]$ we have
\begin{align*}
|\indRad{V_{s_2}}(\psi^{\pm})-\indRad{V_{s_1}}(\psi^{\pm})|\leq\epsilon.
\end{align*}
Hence by Proposition \ref{semicontinuity} we conclude
\begin{align*}
|\indRad{V_{s+\delta}}(\psi^{\pm})-\indRad{V_s}(\psi^{\pm})|\leq\epsilon
\end{align*}
for all $\delta\in [-\delta_0,\delta_0]\cap [s-t_1,t_2-s]$.
As $\psi=\psi^+-\psi^-$ and as $\epsilon>0$ 
was arbitrary this establishes the result.

\textit{ Case 2:}
$x_0\in\spt\mu\cap\{|\psi|>0\}$.
\\
There exists an $r>0$ such that $\oball{x_0}{4r}\subset\{\chi>0\}$.
Then we find an $\epsilon\in (0,1]$ such that
\begin{align*}
\mu(\oball{x_0}{2r})\geq\epsilon
\quad\text{and}\quad
\inf_{\oball{x_0}{2r}}\chi\geq\epsilon.
\end{align*}
By \eqref{jumpingmeasure} we can estimate
\begin{align*}
L(\chi)
\geq\int_{\oball{x_0}{2r}}\chi\ud\mu
\geq\mu(\oball{x_0}{2r})\inf_{\oball{x_0}{2r}}\chi
\geq\epsilon^2.
\end{align*}
Definition of $L$ yields
\begin{align*}
\lim_{\delta\searrow 0}\indRad{V_{s-\delta}}(\chi)\geq\epsilon^2 + \lim_{\delta\searrow 0}\indRad{V_{s+\delta}}(\chi)
\end{align*}
which contradicts $\utder\indRad{V_t}(\chi)|_{t=s}>-\infty$,
so this case actually never occures.
\end{proof}
%
%
%
%
%
\begin{lem}[{\cite[\S 3.5]{MR485012}}]
\label{time_dep_test_fct_bnd_lem}
Consider Setting \ref{bounded_force_setting}
and suppose (B\ref{originalBrakkeflow}) holds.
Then for $s\in [t_1,t_2]$ and $\phi\in\dspace{1}([t_1,t_2],\cdspace{1}(U))$
we have
\begin{align*}
\utder\indRad{V_t}(\phi_t)|_{t=s}
\leq\utder\indRad{V_t}(\phi_s)|_{t=s}
+\limsup_{\delta\to 0}\indRad{V_{s+\delta}}(\tder\phi_t|_{t=s}).
\end{align*}
\end{lem}
%
%
\begin{rem}
\label{Brakke_error}
Here we follow the proof of Brakke's time dependent test function result \cite[\S 3.5]{MR485012}.
Note that Brakke only considers limits from the right.
Also he uses (B\ref{originalBrakkeflow}) to estimate
$\limsup_{\delta\searrow 0}\indRad{V_{s+\delta}}(\tder\phi_t|_{t=s})
\leq\indRad{V_{s}}(\tder\phi_t|_{t=s})$
which is wrong unless $\tder\phi_t|_{t=s}$ is positive,
but this excludes most of the interesting test-functions
in particular his barrier functions.
\end{rem}

%
%
\begin{proof}
For $\delta\in [s-t_1,t_2-s]\setminus\{0\}$ set
\begin{align*}
\xi_s(\delta):=\delta^{-1}(\indRad{V_{s+\delta}}(\phi_{s+\delta})-\indRad{V_{s}}(\phi_{s})).
\end{align*}
Adding and subtracting $\delta^{-1}\indRad{V_{s+\delta}}(\phi_{s})$ and $\indRad{V_{s+\delta}}(\tder\phi_t|_{t=s})$
we can re-arrange terms to obtain
\begin{align}
\label{differencequotinet}
\begin{split}
\xi_s(\delta)
=&\delta^{-1}(\indRad{V_{s+\delta}}(\phi_{s})-\indRad{V_{s}}(\phi_s))
+\indRad{V_{s+\delta}}(\tder\phi_t|_{t=s})\\
&+\delta^{-1}\indRad{V_{s+\delta}}(\phi_{s+\delta}-\phi_{s}-\delta\tder\phi_t|_{t=s}).
\end{split}
\end{align}
Next consider $K:=\bigcup_{t\in[t_1,t_2]}\spt\phi_t$.
By continuity of $\phi$ we see that $K\subset\subset U$.
Then Lemma \ref{uniformmeasurebound} yields an $M\in\real^+$ such that
$\indRad{V_t}(K)\leq M$ for all $t\in [t_1,t_2]$.
This lets us estimate
\begin{align*}
\indRad{V_{s+\delta}}(\phi_{s+\delta}-\phi_{s}-\delta\tder\phi_t|_{t=s})
&=\int_{U}\int_{s}^{s+\delta}(\tder\phi_t|_{t=q}-\tder\phi_t|_{t=s})\ud q\ud\indRad{V_{s+\delta}}\\
&\leq\delta M\sup_{q\in [s-|\delta|,s+|\delta|]}\big|\tder\phi_t|_{t=q}-\tder\phi_t|_{t=s}\big|.
\end{align*}
Thus by the continuity of $\tder\phi_t$ we have
\begin{align*}
\limsup_{\delta\to 0}\delta^{-1}\indRad{V_{s+\delta}}(\phi_{s+\delta}-\phi_{s}-\delta\tder\phi_t|_{t=s})=0.
\end{align*}
Hence taking the $\limsup_{\delta\to 0}$ in \eqref{differencequotinet} establishes the result.
\end{proof}

%
%
\begin{prop}
\label{main_result_bounded}
In Setting \ref{bounded_force_setting} we have
(B\ref{originalBrakkeflow})
implies (B\ref{strongBrakkeflow}) 
\end{prop}
%
%
\begin{proof}
Suppose (B\ref{originalBrakkeflow}) holds.
Consider $s\in [t_1,t_2]$, $\epsilon>0$, and a testfunction 
$\phi\in\dspace{1}((s-\epsilon,s+\epsilon),\cdspace{1}(U,\real^+))$.
First consider the case $\utder\indRad{V_t}(\phi_s)|_{t=s}=-\infty$.
Then Lemma \ref{time_dep_test_fct_bnd_lem}, Lemma \ref{uniformmeasurebound} and continuity of $\tder\phi_t$
imply $\utder\indRad{V_t}(\phi_t)|_{t=s}=-\infty$, so the inequality in (B\ref{strongBrakkeflow}) holds.
Now suppose $\utder\indRad{V_t}(\phi_s)|_{t=s}>-\infty$.
As $\phi$ is positive and differentiable in time we see $\{|\tder\phi_t|_{t=s}|>0\}\subset\{\phi_s>0\}$.
Thus Lemma \ref{almostcontinuity} implies
\begin{align*}
\limsup_{\delta\to 0}\indRad{V_{s+\delta}}(\tder\phi_t|_{t=s})
=\indRad{V_{s}}(\tder\phi_t|_{t=s}).
\end{align*}
Combining this with Lemma \ref{time_dep_test_fct_bnd_lem} and using (B\ref{originalBrakkeflow})
yields the inequality in (B\ref{strongBrakkeflow}).
\end{proof}

%
%
\begin{lem}
\label{Brakke_var_continuity_lem}
Suppose $\BVar$ is upper-continuous on $\subVar$ (see Definition \ref{Variation_defs}),
$\force$ is in $\dspace{0}([t_1,t_2],\dspace{0}(U,\real^{\fulldim}))$
and $(V_t)_{t\in[t_1,t_2]}$ is a Brakke flow.
Let $\phi\in\cdspace{2}(U,\real^+)$
and $s\in [t_1,t_2]$ with $\utder\indRad{V_t}(\phi)|_{t=s}>-\infty$.
Then we have
\begin{align*}
B_s:=\lim_{\delta\to 0}\sup_{t\in (s-\delta,s+\delta)\cap[t_1,t_2]\setminus\badtimes_0}\BVar(V_t,\phi,\force_t)
\leq\BVar(V_s,\phi,\force_s).
\end{align*}
\end{lem}
%
%
\begin{proof}
We will use Definition \ref{Variation_defs}(\ref{upper_continuity}).
Take a sequence $(t_i)_{i\in\nat}$ in $[t_1,t_2]\setminus\badtimes_0$ with $t_i\to s$ and such that
$\BVar(V_i,\phi,\aforce_i)\to B_s$, where we set $V_i:=V_{t_i}\in\subVar$, and $\aforce_i:=\force_{t_i}$.
Also set $\chi:=\phi$, $V_0:=V_s$ and $\aforce_0:=\force_s$.
With Lemma \ref{uniformmeasurebound} we can find $M\in\real^+$
such that \eqref{uniform_measure_bound} holds.
Proposition \ref{almostcontinuity} yields \eqref{measure_convergence}.
Also the continuity of $\force$ implies \eqref{force_convergence}.
Then
\begin{align*}
B_s
=\lim_{i\to\infty}\BVar(V_i,\phi,\aforce_i)
=\limsup_{i\to\infty}\BVar(V_i,\chi,\aforce_i)
\leq\BVar(V_0,\chi,\aforce_0)
=\BVar(V_s,\phi,\force_s).
\end{align*}
\end{proof}
%
%
%
%
\begin{lem}
\label{measurability}
Suppose $\BVar$ is upper-continuous on $\subVar$ (see Definition \ref{Variation_defs}),
$\force$ is in $\dspace{0}([t_1,t_2],\dspace{0}(U,\real^{\fulldim}))$
and (B\ref{originalBrakkeflow}) holds.
Consider $\phi\in\dspace{1}([t_1,t_2],\cdspace{2}(U,\real^+))$.
Then $s\to\BVar(V_s,\phi_s,\force_s)+\indRad{V_s}(\tder\phi_t|_{t=s})$
is measurable on $[t_1,t_2]$.
\end{lem}
\begin{proof}
For $i\in\nat$ set $\tau_i:=(t_2-t_1)/i$
and $a_j:=t_1+j\tau_i$, $j=0,\ldots i-1$.
Define $\psi^i,\phi^i\in\dspace{1}([t_1,t_2],\cdspace{2}(U,\real^+))$
by $\psi^i_s=\tder\phi_t|_{t=a_j}$ and $\phi^i_s=\phi_{a_j}$ for $s\in [a_j,a_{j+1})$.
Then for arbitrary $s\in [t_1,t_2]$ we can calculate 
\begin{align}
\label{approx_conv}
\lim_{i\to\infty}\|\psi^i_s-\tder\phi_t|_{t=s}\|_{\dspace{0}(U)}=0
\quad\text{and}\quad
\lim_{i\to\infty}\|\phi^i_s-\phi_s\|_{\dspace{2}(U)}=0.
\end{align}
Proposition \ref{semicontinuity} and Lemma \ref{Brakke_var_continuity_lem}
yield $s\to\indRad{V_s}(\psi^i_s)$ 
and $s\to\BVar(V_s,\phi_s^i,\force_s)$ are measurable.
Hence $s\to\limsup_{i\to\infty}(\BVar(V_s,\phi_s^i,\force_s)+\indRad{V_s}(\psi^i_s))$
is measurable.
Convegrence \eqref{approx_conv} yields $\lim_{i\to\infty}\indRad{V_s}(\psi^i_s)=\indRad{V_s}(\psi_s)$,
thus it suffices to show
\begin{align}
\label{Brakkevar_conv}
\lim_{i\to\infty}\BVar(V_s,\phi_s^i,\force_s)
=\BVar(V_s,\phi_s,\force_s)
\end{align}
for $\Lp{1}$-a.e.\ $s\in[t_1,t_2]$.

Consider a function $\chi\in\cdspace{2}(U,[0,1])$ with $\spt\phi_s\subset\{\chi=1\}$ for all $s\in[t_1,t_2]$.
By Proposition \ref{semicontinuity} and Lemma \ref{fundamentalinequality}
there exists a set $\badtimes_1\subset [t_1,t_2]$
such that
$\utder\indRad{V_t}(\chi)|_{t=s}>-\infty$ 
for all $s\in[t_1,t_2]\setminus\badtimes_1$
and $\LM{1}(\badtimes_1)=0$.
As (B\ref{originalBrakkeflow}) holds this yields
$\int_{U}\chi|\genMCV_s|^2\indRad{V_s}<\infty$
for all $s\in[t_1,t_2]\setminus\badtimes_1$.
In view of Definition \ref{Variation_defs}(\ref{strong_continuity}) the convergence \eqref{approx_conv} 
yields \eqref{Brakkevar_conv}, which completes the statement.
\end{proof}
This completes the ingredients for the proof of the \hyperlink{main_result}{Main Result}.
%
%
\begin{proof}[proof of the \hyperlink{main_result}{Main Result}]
(B\ref{originalBrakkeflow}) implies (B\ref{strongBrakkeflow})
by Proposition \ref{main_result_bounded}.
\\
(B\ref{strongBrakkeflow}) implies (B\ref{Tonegawaflow_variant})
follows from Lemma \ref{fundamentalinequality} and Lemma \ref{measurability}.
\\
(B\ref{Tonegawaflow_variant}) implies (B\ref{Tonegawaflow})
by Corollary \ref{implications_of_weak_definitions}.
\\
(B\ref{Tonegawaflow}) implies (B\ref{weakBrakkeflow}) is clear.
\\
(B\ref{weakBrakkeflow}) implies (B\ref{originalBrakkeflow}) 
follows from Lemma \ref{Brakke_var_continuity_lem} and Remark \ref{upper_origianl_BF}.
\end{proof}
\begin{cor}
\label{equal_variations}
Consider the case $\subVar=\IVar_{\vdim}(U)$
and $\force\in\dspace{0}([t_1,t_2],\dspace{0}(U,\real^{\fulldim}))$.
Recall Definition \ref{Tonegawas_variation} and suppose 
$\BVar(V,\phi,\aforce)=\TBVar(V,\phi,\aforce)$ whenever $V\in\IVar_{\vdim}(U)$.
Let $(V_t)_{t\in[t_1,t_2]}$ be a Brakke flow.
Then
\renewcommand{\labelenumi}{(\arabic{enumi})}
\begin{enumerate}
\item 
All the characterisations (B1)-(B5) hold.
\item
For $s\in [t_1,t_2]$ such that $V_s\in\Var_{\vdim}(U)\setminus\IVar_{\vdim}(U)$
we have
\begin{align*}
\utder\indRad{V_t}(\phi)|_{t=s}=-\infty
\quad\text{for all }\phi\in\cdspace{2}(U,\real^+).
\end{align*}
\end{enumerate}
\end{cor}
\begin{proof}
Corollary \ref{changed_variation_cor} yields that $\BVar$ is upper-continuous on $\IVar_{\vdim}(U)$. 
Thus the \hyperlink{main_result}{Main Result} establishes (1).
In particular (B\ref{weakBrakkeflow}) holds.
Then also (B\ref{weakBrakkeflow}) with $\BVar$ replaced by $\TBVar$ holds.
In view of Proposition \ref{proper_Brakke_variation}
we can apply the \hyperlink{main_result}{Main Result} with $\BVar$ replaced by $\TBVar$.
Thus (B\ref{originalBrakkeflow}) holds with $\BVar$ replaced by $\TBVar$.
The definition of $\TBVar$ then establishes (2).
\end{proof}

\begin{appendix}
\section{Appendix}
\label{appendix}
Here we show two very fundamental lemmas regarding the growth of functions with bounded upper derivative.
Also we prove an $\Lp{2}$-approximation Lemma for $\Lp{1}$-functions.

\begin{lem}
\label{leftmonotonicitylem}
Let $a,L\in\real$, $b\in (a,\infty)$
and $f:[a,b]\to\mathbb{R}$.
Suppose
\begin{align*}
\lefttder f(t)|_{t=s}\leq L
\quad\text{for all } s\in (a,b],
\\
\limsup_{h\searrow 0}f(t+h)\leq f(t)
\quad\text{for all } t\in [a,b).
\end{align*}
Then $f(b)- f(a)\leq L(b-a)$.
\end{lem}
%
%
\begin{proof}
First consider the case $a=0$, $f(0)=0$ and $L=0$.
For $\epsilon\in (0,1)$ define
\begin{align*}
I_{\epsilon}&:=\{s\in [0,b]:
f(t) <  \epsilon (1+t)\quad\text{for all } t\in [0,s]\}.
\end{align*}
We directly see that $I_{\epsilon}\neq\emptyset$ 
as $0\in I_{\epsilon}$.

Consider some $s\in I_{\epsilon}$, $s<b$.
There exists an $\epsilon_s\in (0,1)$ such that we have
$f(s)+\epsilon_s < \epsilon (1+s)$.
Also there exists a $\tau_s\in (0,b-s)$ such that
for all $t\in (s,s+\tau_s)$ we have
$f(t)\leq \limsup_{\delta\searrow 0}f(s+\delta)+\epsilon_s$.
Hence by the upper continuity from the right we can estimate
\begin{align*}
f(t)
&\leq \limsup_{\delta\searrow 0}f(s+\delta)+\epsilon_s
\leq f(s)+\epsilon_s
<\epsilon (1+s)<\epsilon (1+t)
\end{align*}
for all $t\in (s,s+\tau_s)$.
So $[0,s+\tau_s)\subset I_{\epsilon}$.

Now consider a sequence $(s_m)_{m\in\mathbb{N}}$
with $s_m\in I_{\epsilon}$, $s_{m+1}> s_m$ for all $m\in\mathbb{N}$
and $\lim_{m\to\infty}s_m=s_0$ for some $s_0\in [0,b]$.
As $\lefttder f(t)|_{t=s_0}\leq 0$ there exists an $m\in\mathbb{N}$
such that $(s_0-s_m)^{-1}(f(s_0)-f(s_m))\leq \epsilon$.
Hence we can estimate
\begin{align*}
f(s_0)
\leq f(s_m)+\epsilon (s_0-s_m)
<\epsilon(1+s_m)+\epsilon (s_0-s_m)
=\epsilon(1+s_0).
\end{align*}
This yields $s_0\in I_{\epsilon}$.

So we showed $I_{\epsilon}$ is open, closed and non-empty in $[0,b]$, thus $I_{\epsilon}=[0,b]$.
In particular $b\in I_{\epsilon}$ for all $\epsilon\in (0,1)$.
Letting $\epsilon\searrow 0$ then establishes the result.
For general $a,D,f(a)\in\real$ consider $h(t):=f(t+a)-f(a)-Lt$ for $t\in [0,b-a]$.
Applying the established case yields
$h(b-a)\leq 0$ and we conclude the result.
\end{proof}
\begin{lem}
\label{fundamentalinequality}
Let $a,L\in\real$, $b\in (a,\infty)$
and $f:[a,b]\to\mathbb{R}$.
Suppose
\begin{align*}
\utder f(t)|_{t=s}\leq L
\quad\text{for all } s\in [a,b].
\end{align*}
Then we have $s\to\utder f(t)|_{t=s}$ is integrable and
\begin{align*}
f(b)-f(a)\leq\int_{a}^{b}\utder f(t)|_{t=s}\ud s.
\end{align*}
\end{lem}
%
%
\begin{proof}
First consider the case $a=0$, $f(0)=0$ and $L=0$.
Then Lemma \ref{leftmonotonicitylem} yields 
that $f$ is non-increasing, thus measurable.
Set $g(s):=\utder f(t)|_{t=s}$
then $g$ is measurable and non-positive,
hence integrable. (See \cite[\S 2.3.2]{MR0257325})

For $m\in\nat$ define
\begin{align*}
g_m(t):=m\left(f(t)-f(t-m^{-1})\right).
\end{align*}
For arbitrary $\epsilon> 0$ consider $m>\epsilon^{-1}$.
Then the monotonicity of $f$ and $f(0)=0$ yield
\begin{align*}
\int_{\epsilon}^{b}g_m(t)\ud t
&=m\int_{b-m^{-1}}^{b}f(t)\ud t 
- m\int_{\epsilon-m^{-1}}^{\epsilon}f(t)\ud t
\geq f(b).
\end{align*}
As the $g_m$ are non-positive
we can use Fatou's Lemma to obtain
\begin{align*}
f(b)
\leq
\limsup_{m\to \infty}\int_{\epsilon}^{b}g_m(s)\ud s
\leq
\int_{\epsilon}^{b}\left(\limsup_{m\to \infty}g_m(s)\right)\ud s
\end{align*}
for every $\epsilon>0$.
By definition of $g$ and $g_m$ we clearly have
$\limsup g_m(s)\leq g(s)$ for all $s\in [a,b]$.
So we conclude
$f(b)\leq\int_{\epsilon}^{b}g(s)\ud s$
for every $\epsilon>0$,
and letting $\epsilon\searrow 0$
we can use the monotone convergence theorem to obtain
the result in the special case.

For general $a,D,f(a)\in\real$ consider $h(t):=f(t+a)-f(a)-Lt$ for $t\in [0,b-a]$.
Applying the established case yields
$h(b-a)\leq\int_{0}^{b-a}\utder h(t)|_{t=s}\ud s$ and we conclude the result.
\end{proof}

%
%
\begin{lem}
\label{L2_approx_lem}
Consider a Radon measure $\mu$ on some open set $W\subset\real^{\fulldim}$.
Let $H\in\Lp{1}((W,\mu),\real^{\fulldim})$, suppose $\mu(W)<\infty$
and set  
\begin{align*}
S:=\sup\left\{\int_{W}H\cdot X\,\ud\mu,
\; X\in\dspace{\infty}(W,\real^{\fulldim}),
\,\|X\|_{\Lp{2}((W,\mu),\real^{\fulldim})}\leq 1\right\}.
\end{align*}
Then we have $\|H\|_{\Lp{2}((W,\mu),\real^{\fulldim})}=S$.
Note that this may state $\infty=\infty$.
\end{lem}
\begin{proof}
We set $\Lp{p}:=\Lp{p}((W,\mu),R)$,
where $R$ is $\real$ or $\real^{\fulldim}$ which will be clear from the context.
The inequality "$\geq$" follows immediately from H\"olders inequality.
For "$\leq$" consider some $L\in\real^+$ and set
\begin{align*}
h(x):=|H(x)|,
\quad
h_L(x):=\min\{h(x),L\},
\quad
\nu(x):=
\begin{cases}
\frac{H(x)}{|H(x)|}
&\text{if }|H(x)|>0\\
0 &\text{else}  
\end{cases}
\end{align*}
for $x\in W$.
By the finite measure of $W$ we see $h_L\nu\in\Lp{2}$.
We may assume $\|h_L\|_{\Lp{2}}>0$ or else $h_L\equiv 0$
thus $H\equiv 0$ wich trivially establishes the estimate.
Let $\epsilon\in (0,1)$ with $\epsilon\leq\|h_L\|_{\Lp{2}}/2$.
Then there exists an $X_{\epsilon}\in\dspace{1}(W,\real^{\fulldim})$ with
\begin{align*}
\|X_{\epsilon}-h_L\nu\|_{\Lp{2}}\leq\epsilon^2,
\quad
\left|1-\frac{\|X_{\epsilon}\|_{\Lp{2}}}{\|h_L\|_{\Lp{2}}}\right|\leq\epsilon,
\quad
\|X_{\epsilon}\|_{\Lp{2}}\geq\frac{\|h_L\|_{\Lp{2}}}{2},
\quad
\sup|X_{\epsilon}|\leq 2L.
\end{align*}
Here the first estimate resembles an $\Lp{2}$-approximation,
see \cite[Lem.\ 7.4]{MR1196160}.
Then $\epsilon\leq\|h_L\|_{\Lp{2}}/2$ yields the second and third estimate.
The $\sup$-bound can be realized by cutting off $X_{\epsilon}$.

Now we have
\begin{align*}
\|h_L\|_{\Lp{2}}^2
&=
\int_U\left(h X_{\epsilon}\cdot\nu
+ h_L(h_L-X_{\epsilon}\cdot\nu)
- X_{\epsilon}\cdot\nu(h-h_L)\right)\ud\indRad{V_0}.
\end{align*}
Set $A:=\{X_{\epsilon}\cdot\nu<0\}\cap\{h>L\}$.
Then by H\"older's inequality and $h=h_L$ on $\{h\leq L\}$ we can estimate
\begin{align*}
\|h_L\|_{\Lp{2}}^2
\leq &
\int_U H\cdot X_{\epsilon}\ud\mu
+\epsilon\|h_L\|_{\Lp{2}}
+2L\int_{A} |h-L|\,\ud\mu.
\end{align*}
Deviding by $\|X_{\epsilon}\|_{\Lp{2}}$
we arrive at
\begin{align}
\label{cuttet_estimate2}
\begin{split}
(1-\epsilon)\|h_L\|_{\Lp{2}}
\leq
S+2\epsilon
+\frac{4L}{\|h_L\|_{\Lp{2}}}\int_{A} |h|\,\ud\mu.
\end{split}
\end{align}
For the measure of $A$ estimate
\begin{align*}
\mu(A)
\leq\frac{1}{L^2}\int_A|L-X_{\epsilon}\cdot\nu|^2\ud\mu
=\frac{1}{L^2}\int_A|h_L-X_{\epsilon}\cdot\nu|^2\ud\mu
\leq\frac{\epsilon^2}{L^2}.
\end{align*}
As $h$ is in $\Lp{1}$ letting $\epsilon\searrow 0$ in \eqref{cuttet_estimate2}
implies $\|h_L\|_{\Lp{2}}\leq S$. (See \cite[\S 2.4.11]{MR0257325})
As $L\in\real^+$ was arbitrary and by definition of $h_L$
the monotone convergence theorem establishes the result.
\end{proof}

\end{appendix}

\bibliography{Lahiris_Bib}{}
\bibliographystyle{alpha}

\noindent
Ananda Lahiri\\
Max Planck Institute for Gravitational Physics\\
(Albert Einstein Institute)\\
Am M\"uhlenberg 1\\
D-14476 Potsdam-Golm, Germany

\end{document}